\documentclass[12pt]{amsart}

\newtheorem{theorem}{Theorem}[section]
\newtheorem{lemma}[theorem]{Lemma}
\newtheorem{proposition}[theorem]{Proposition}
\newtheorem{corollary}[theorem]{Corollary}
\newtheorem{claim}{Claim}[theorem]

\theoremstyle{definition}

\theoremstyle{remark}

\numberwithin{equation}{section}

\usepackage{latexsym}
\usepackage{amssymb}
\usepackage{amsmath}
\usepackage{amscd}
\usepackage{mathrsfs}

  \DeclareMathOperator{\Span}{{\rm span}}
  \DeclareMathOperator{\GL}{{\rm GL}} \DeclareMathOperator{\SL}{{\rm
      SL}} 
  \DeclareMathOperator{\End}{{\rm End}}

  \DeclareMathOperator{\norL}{{\rm N}_L}
  
  \DeclareMathOperator{\noroneL}{{\rm N}^1_L}
  \DeclareMathOperator{\cen}{{\rm Z}_G} 
  \DeclareMathOperator{\supp}{{\rm supp}} 
  \DeclareMathOperator{\Stab}{{\rm Stab}} 
  \DeclareMathOperator{\Ad}{{\rm Ad}} 
  \DeclareMathOperator{\diag}{\rm{diag}} 
  \DeclareMathOperator{\Lie}{{\rm Lie}} 
  \DeclareMathOperator{\Cc}{{\rm C}_{{\rm c}}} 

   \renewcommand{\phi}{\varphi}

 \newcommand{\vect}[1]{{\boldsymbol{#1}}}

  \newcommand{\vx}{\vect{x}} \newcommand{\vy}{\vect{y}}

  \newcommand{\vxi}{\vect{\xi}}
  \newcommand{\vpsi}{\vect{\psi}}
  \newcommand{\vzeta}{\vect{\zeta}}
  \newcommand{\veta}{\vect{\eta}}
  \newcommand{\vlambda}{\vect{\lambda}}
  \newcommand{\vzero}{\vect{0}}

  \newcommand{\field}[1]{\mathbb{#1}} 
  \newcommand{\R}{\field{R}} \newcommand{\N}{\field{N}}
  \newcommand{\Q}{\field{Q}} \newcommand{\Z}{\field{Z}}

  \providecommand{\abs}[1]{\lvert#1\rvert}
  \providecommand{\Abs}[1]{\left\lvert #1 \right\rvert}
  \providecommand{\norm}[1]{\lVert#1\rVert}

  \providecommand{\trn}[1]{{\,^{\bf t}\!#1}}

  \renewcommand{\setminus}{\smallsetminus}

  \newcommand{\inv}{^{-1}}

  \newcommand{\tcup}[1]{\textstyle{\bigcup_{#1}}\,}

  \newcommand{\toinfty}{\stackrel{i\to\infty}{\longrightarrow}}

  \newcommand{\la}[1]{\mathfrak{\lowercase{#1}}}

  \newcommand{\cA}{\mathcal{A}} \newcommand{\sA}{\mathscr{A}}

  \newcommand{\cB}{\mathcal{B}} 

   \newcommand{\cC}{\mathcal{C}}

  \newcommand{\cD}{\mathcal{D}}

  \newcommand{\cE}{\mathcal{E}}

  \newcommand{\cF}{\mathcal{F}} \newcommand{\sF}{\mathscr{F}}

   \newcommand{\sH}{\mathscr{H}}

  \newcommand{\cK}{\mathcal{K}}

  \newcommand{\cN}{\mathcal{N}}

  \newcommand{\cO}{\mathcal{O}}

\newcommand{\lml}{{L/\Lambda}}
\newcommand{\dth}{D.Th.}
\newcommand{\A}{{(A)}}   
\newcommand{\B}{{(B)}}
\newcommand{\Amu}{{(A$\mu$)}}
\newcommand{\Bmu}{{(B$\mu$)}}

\begin{document}

\title[Equidistribution and Dirichlet's theorem]{Equidistribution of
  expanding translates of curves and Dirichlet's theorem on
  Diophantine approximation}

\author{Nimish A. Shah}

\address{Tata Institute of Fundamental Research, Mumbai 400005,
  India}

\email{nimish@math.tifr.res.in}


\begin{abstract}
  We show that for almost all points on any analytic curve on $\mathbb{R}^k$
  which is not contained in a proper affine subspace, the Dirichlet's
  theorem on simultaneous approximation, as well as its dual result
  for simultaneous approximation of linear forms, cannot be
  improved. The result is obtained by proving asymptotic
  equidistribution of evolution of a curve on a strongly unstable leaf
  under certain partially hyperbolic flow on the space of unimodular
  lattices in $\mathbb{R}^{k+1}$. The proof involves ergodic properties of
  unipotent flows on homogeneous spaces.


\end{abstract}

\maketitle

\section{Introduction}
The Dirichlet's theorem on simultaneous approximation of any $k$ real
numbers $\xi_1,\dots,\xi_k$ says the following:
  \begin{itemize}
  \item[{\A}] For any positive integer $N$ there exist integers
  $q_1,\dots,q_k,p$ such that
    \begin{equation*}
       \abs{q_1\xi_1+\dots+q_k\xi_k-p}\leq N^{-k} \quad \text{and}\quad
       0<\max_{1\leq i\leq k}\abs{q_i}\leq N;
    \end{equation*}
  \item[{\B}] For any positive integer $N$ there exist integers
    $q,p_1,\dots,p_k$ such that
      \begin{equation*}
        \abs{q\xi_i-p_i}\leq N^{-1}
        \quad\text{and}\quad \max_{1\leq i\leq k}\abs{q}\leq N^k.
      \end{equation*}
    \end{itemize}
    
    After \cite{Davenport+Schmidt:Dirichlet}, we say that the \dth\
    {\A} (respectively, {\B}) cannot be improved for
    $\vxi=(\xi_1,\dots,\xi_k)\in\R^k$ if for any $0<\mu<1$ the
    following holds:
\begin{itemize}  
\item[{\Amu}] There are infinitely many positive integers $N$ for
  which the pair of inequalities 
\begin{equation*}
      \abs{(q_1\xi_1+\dots+q_k\xi_k)-p}\leq \mu N^{-k} \quad \text{and} \quad
      0<\max_{1\leq i\leq k} \abs{q_i}\leq \mu N
    \end{equation*} 
    are insoluble in integers $q_1,\dots,q_k,p$ (respectively,
  \item[{\Bmu}] there are infinitely many positive integers $N$ for
    which the pair of inequalities 
  \begin{equation*}
    \abs{q\xi_i-p_1}\leq \mu N^{-1}
        \quad\text{and}\quad 
        0<\abs{q}\leq \mu N^k
  \end{equation*}
  are insoluble in integers $q,p_1,\dots,p_k$).
\end{itemize}

In \cite{Davenport+Schmidt:Dirichlet}, Davenport and Schmidt proved
that {\dth}\ {\A} and {\B} cannot be improved for almost all
$\vxi\in\R^k$.

One says that {\dth}\ {\A} (respectively, {\B}) cannot be {\em
  $\mu$-improved\/} for $\xi\in\R^k$ if {\Amu} (respectively, {\Bmu})
holds. In \cite{DS:curve} Davenport and Schmidt showed that {\dth}\
{\A} cannot be $(1/4)$-improved for the pair $(\xi,\xi^2)$ for almost
all $\xi\in\R$. This result was generalized by
Baker~\cite{Baker:curves} for almost all points on `smooth curves' in
$\R^2$, by Dodson, Rynne and Vickers~\cite{Dodson:manifolds} for
almost all points on `higher dimensional curved submanifolds' of
$\R^k$, and by Bugeaud~\cite{Bugeaud:poly} for almost all points on
the curve $(\xi,\xi^2,\dots,\xi^k)$; in each case {\Amu} holds for
some small value of $\mu<1$ depending on the curvature of the smooth
submanifold. Their proofs typically involve the technique of regular
system introduced in \cite{Davenport+Schmidt:Dirichlet}.

Recently the problem was recast in the language of flows on
homogeneous spaces by Kleinbock and
Weiss~\cite{Kleinbock+Weiss:Dirichlet} using observations due to
Dani~\cite{Dani:div}, as well as Kleinbock and
Margulis~\cite{Klein+Mar:Annals98}. In
\cite{Kleinbock+Weiss:Dirichlet} it was shown that {\dth}\ (A) and (B),
as well as its various generalizations, cannot be $\mu$-improved for
almost all points on any {\em non-degenerate curve\/} on $\R^k$ for
some small $\mu<1$ depending on the curve. In this article, we shall
strengthen such results for all $0<\mu<1$:

  \begin{theorem}
    \label{thm:Dirichlet}
    Let $\phi:[a,b]\to \R^{k}$ be an analytic curve such that its
    image is not contained in a proper affine subspace. Then
    Dirichlet's theorem (A) and (B) cannot be improved for $\phi(s)$
    for almost all $s\in [a,b]$.
  \end{theorem}
  
  This result will be deduced from a result about limiting
  distributions of certain expanding sequence of curves on the space
  of lattices in $\R^{k+1}$. A refinement of
  Theorem~\ref{thm:Dirichlet} is obtained in
  Theorem~\ref{thm:Dirichlet-2}. 

  \subsubsection{Notation} Let $G=\SL(n,\R)$, $n\geq 2$. For $t\in\R$ and
  $\vxi=(\xi_1,\dots,\xi_{n-1})\in \R^{n-1}$, define 
\begin{gather}
   \label{eq:50} 
   a_t={\left[{\begin{smallmatrix} e^{(n-1)t}\\ & e^{-t}\\ & & \ddots \\ &
         & & e^{-t}
       \end{smallmatrix}} 
    \right]} 
\quad \text{and} \quad 
a'_t={\left[\begin{smallmatrix} e^t\\ &\ddots \\ & & e^t \\ &&& e^{-(n-1)t}
       \end{smallmatrix} \right]} \\
u(\vxi)={\left[\begin{smallmatrix} 1 & \xi_1&\dots&\xi_{n-1}\\ & 1 \\ & &
    \ddots \\ &&& 1
  \end{smallmatrix}
\right]}
\quad \text{and} \quad 
u'(\vxi)={\left[\begin{smallmatrix} 1 &      &      &\xi_{n-1}\\ 
                               &\ddots&      &\vdots \\ 
                               &      & 1    &\xi_{1}\\ 
                               &      &      & 1
               \end{smallmatrix}
\right]}.
\end{gather}

The main goal of this article is to prove the following:

\begin{theorem}
\label{thm:main}  
Let $\phi:[a,b]\to \R^{n-1}$ be an analytic curve whose image is not
contained in a proper affine subspace. Let $\Gamma$ be a
lattice in $G$. The for any $x_0\in G/\Gamma$ and any bounded
continuous function $f$ on $G/\Gamma$,
    \begin{equation}
      \label{eq:9}
      \lim_{t\to\infty} \frac{1}{\abs{b-a}}\int_a^b
      f({a_t}u(\phi(s))x_0)\,ds=\int_{G/\Gamma} f\,d\mu_G, 
    \end{equation}
    where $\mu_G$ is the $G$-invariant probability measure on
    $G/\Gamma$. Similarly,
\begin{equation}
      \label{eq:9b}
      \lim_{t\to\infty} \frac{1}{\abs{b-a}}\int_a^b
      f(a'_tu'(\phi(s))x_0)\,ds=\int_{G/\Gamma} f\,d\mu_G.
    \end{equation}
  \end{theorem}
 
  This result corroborates \cite[\S4.3]{Goro:AIM},
  \cite[\S4.1]{Kleinbock+Weiss:Dirichlet}, and answers
  \cite[Question~6.1]{Shah:son1}. In fact, we will prove the following
  more general statement.

\begin{theorem}
  \label{thm:action}
  Let $L$ be a Lie group and $\Lambda$ a lattice in $L$. Let
  $\rho:G\to L$ be a continuous homomorphism. Let $x_0\in L/\Lambda$
  and $H$ be a minimal closed subgroup of $L$ containing $\rho(G)$
  such that the orbit $Hx_0$ is closed, and admits a finite
  $H$-invariant measure, say $\mu_H$. Then for any bounded continuous
  function $f$ on $L/\Lambda$ the following holds:
    \begin{equation}
      \label{eq:9L}
      \lim_{t\to\infty} \frac{1}{\abs{b-a}}\int_a^b
      f(\rho(a_{t}u(\phi(s)))x_0)\,ds=\int_{Hx_0} f\,d\mu_H. 
    \end{equation}
\end{theorem}

The first part of Theorem~\ref{thm:main} follows from
Theorem~\ref{thm:action} by taking $L=G$, $\rho$ the identity map, and
$\Lambda=\Gamma$; in this case $H=G$.

Define an automorphism $\sigma:G\to G$ as follows: let $s_k$ be the
permutation on the $n$-coordinates of $\R^n$ which exchanges the
$i$-th and the $(n+1-i)$-th coordinates for $1\leq i\leq n$, and we
let $\sigma(g)=s_k(\trn{g}\inv)s_k\inv$ for all $g\in G$. Then
$\sigma(a_t)=a'_t$ and $\sigma(u(\vxi))=u'(\xi)$. Now the second part
of Theorem~\ref{thm:main} follows from Theorem~\ref{thm:action} by
taking $L=G$, $\rho=\sigma$ and $\Lambda=\Gamma$; we observe that
$H=L$ in this case.

\subsubsection{} \label{sec:Dir2} Next we take $L=G\times G$,
$\Gamma=\SL(n,\Z)$, $\Lambda=\Gamma\times\Gamma$ and
$\rho(g)=(g,\sigma(g))$ for all $g\in G$. Note that
$\sigma(\Gamma)=\Gamma$, and hence $\rho(\Gamma)\subset\Lambda$ and
$\rho(\Gamma)$ is a lattice in $\rho(G)$. Therefore if we let
$x_0=\Lambda$, then $\rho(G)x_0$ is closed and admits a finite
$\rho(G)$-invariant measure; in other words, we have $H=\rho(G)$. Now
using Theorem~\ref{thm:action}, in the next section we will deduce the
following enhancement of Theorem~\ref{thm:Dirichlet}

\begin{theorem}
  \label{thm:Dirichlet-2}
  Let $\phi=(\phi_1,\dots,\phi_k):I=[a,b]\to \R^k$ be an analytic curve
  whose image is not contained in a proper affine subspace. Let $\cN$
  be an infinite set of positive integers. Then for almost every $s\in
  I$ and any $\mu<1$, there exist infinitely many $N\in\cN$ such that
  both the following pairs of inequalities are simultaneously
  insoluble:
\begin{equation}
  \label{eq:41}
  \abs{q_1\phi_1(s)+\dots+q_k\phi_k(s)-p}\leq \mu N^{-k} \quad \text{and} \quad
  \max_{1\leq i\leq k}\abs{q_i}\leq \mu N
\end{equation}
for $(q_1,\dots,q_k)\in \Z^k\setminus\{\vzero\}$ and
$p\in\Z$; and
\begin{equation}
  \label{eq:51}
  \max_{1\leq i\leq k}\abs{q\phi_i(s)-p_i}\leq \mu N^{-1}
  \quad\text{and}\quad 
  \abs{q}\leq \mu N^k
\end{equation}
for $q\in\Z\setminus\{0\}$ and $(p_1,\dots,p_k)\in\Z^k$.
\end{theorem}

\subsection{Sketch of the proof of Theorem~\ref{thm:action}} \label{sec:sketch}
Let $I=[a,b]$. We will treat $G$ as a subset of $L$ via the
homomorphism $\rho$. We consider the normalized parameter measure, say
$\nu$, on the segment $\{u(\phi(s))x_0:s\in I\}$ on $\lml$. Take any
sequence $t_i\to\infty$. Let $a_{t_i}\nu$ denote the translate of
$\nu$ concentrated on the curve $a_{t_i}u(\phi(I))x_0$. By
Dani-Margulis criterion for nondivergence, if for every compact set
$\cF\subset\lml$, we have
\begin{equation}
  a_{t_i}\nu(\cF):=\frac{1}{\abs{b-a}}\abs{\{s\in[a,b]:a_{t_i}u(\phi(s))x_0\in\cF\}}\to 0 \quad \text{as $i\to\infty$},
\end{equation}
then the following algebraic condition holds: There exists a finite
dimensional representation $V$ of $G$ and a nonzero vector $v\in V$
such that
\begin{equation}
  \label{eq:102}
\lim_{i\to\infty} a_{t_i}u(\phi(s))v\to 0,  \quad\forall s\in I.
\end{equation}

The `Basic Lemma' (Proposition~\ref{prop:main}), which is one of the
main tools developed in this article, says that if $\{\phi(s):s\in I\}$
affinely spans $\R^{n-1}$ then \eqref{eq:102} cannot hold.

Therefore using the Dani-Margulis criterion we deduce that given any
$\epsilon>0$ there exists a large compact set $\cF\subset\lml$ such
that $(a_{t_i}\nu)(\cF)\geq 1-\epsilon$. As a consequence, there
exists a probability measure $\mu$ on $\lml$ such that after passing
to a subsequence, $a_{t_i}\nu\toinfty\mu$ with respect to the
weak-$\ast$ topology.

At this stage, we note that $u(\phi(I))x_0$ is contained in a strongly
unstable leaf for the action of $a_t$ on $\lml$. Then for each $s_0\in
I$ if $\dot\phi(s_0)$ denotes the derivative of $\phi$ at $s_0$, then
\begin{equation}
  \label{eq:113}
\phi(s)-\phi(s_0)=(s-s_0)\dot\phi(s_0)+O((s-s_0)^2).
\end{equation}
Therefore for any large $t>0$, the translated curve
\begin{equation}
  \label{eq:114}
  \{a_tu(\phi(s))x_0=
  u(e^{nt}(\phi(s)-\phi(s_0)))(a_tu(\phi(s_0))x_0):\abs{s-s_0}<\delta_t\}
\end{equation}
stays very close to the unipotent trajectory
\begin{equation}
  \label{eq:unip}
  \{u(e^{nt}(s-s_0)\dot\phi(s_0))(a_tu(\phi(s_0))x_0):\abs{s-s_0}<\delta_t\},
\end{equation}
if we choose $\delta_t>0$ such that $e^{nt}\delta_t\to\infty$, but
$e^{nt}\delta_t^2\to 0$ as $t\to\infty$.  Note that the length of the
unipotent trajectory \eqref{eq:unip} is about
$e^{nt}\delta_t\norm{\dot\phi(s_0)}$ and hence it is very long. Our
difficultly is that the direction $\dot\phi(s_0)$ of the flow varies
with $s_0$.

In order to take care of this problem, instead of translating the
original curve, we twist it by $z(s)\subset \cen(A)\cong A\times
\SL(n-1,\R)$ so that $z(s)u(\dot\phi(s))z(s)\inv=u(e_1)$ for all $s\in
I$; here $A=\{a_t:t\in\R\}$, $e_1$ is a fixed nonzero vector in
$\R^{n-1}$ and we assume that $\dot\phi(s)\neq 0$ for all $s\in I$. We
take another curve: $\{z(s)u(\phi(s))x_0:s\in I\}$ and associate a
normalized parameter measure $\nu'$ on it. Since $z(I)$ is contained
in a compact set, we conclude that after passing to a subsequence
$a_{t_i}\nu'$ converges to a probability measure $\nu'$ as
$i\to\infty$. Now one can show that $\nu'$ is invariant under the flow
$\{u(se_1):s\in \R\}$. Then we will be in the situation where we can
use Ratner's theorem and the so called `linearization
technique'. Using both and the `Basic Lemma', we will show that there
exist a nonzero vector $v\in V$, which is algebraically associated to
$x_0$, and a curve $I\ni s\mapsto w(s)\in V$ such that
\begin{equation}
  \label{eq:103}
  \lim_{t\to\infty} a_{t}u(\phi(s))v\to w(s), \quad \forall s\in I. 
\end{equation}
Again using an extension of the `Basic lemma'
(Corollary~\ref{cor:curve-main}) we show that $v$ is fixed by
$G$. From this very restrictive situation, we further deduce that the
measure $\nu'$ is $H$-invariant, where $H$ is the smallest closed
subgroup of $L$ containing $G$ such that the orbit $Hx_0$ is
closed. Since the modification of $\nu$ to obtain $\nu'$ was only by
elements centralizing $A$ in $G$, and since we have shown that $\nu'$
is invariant under $\cen(A)$, we obtain that $\nu=\nu'$, and hence
$\nu$ is $H$-invariant. This will prove Theorem~\ref{thm:action}.

\subsection{Variations of the equidistribution result}
\label{sec:variations}
\subsubsection{Expanding translates of any curve} The following
 variation of Theorem~\ref{thm:action} is more natural.

\begin{theorem}
  \label{thm:curve}
  Let $\vpsi:[a,b]\to \SL(n,\R)$ be an analytic curve such that the image
  of the first row of this curve is not contained in a proper subspace
  of $\R^n$.  Let the notation be as in the statement of
  Theorem~\ref{thm:action}. Then
\begin{equation}
  \label{eq:62}
  \frac{1}{b-a}\int_a^b
  f(\rho(a_t\vpsi(s))x_0)\,d s \stackrel{t\to\infty}{\longrightarrow} 
\int_{Hx_0} f\,d\mu_H.
\end{equation}
\end{theorem}

The above result will be deduced from Theorem~\ref{thm:main} by using
\cite[Lemma~5.1]{Shah:son1}.

\subsubsection{Uniform versions} 
First we state the basic uniform version of Theorem~\ref{thm:main}.

\begin{theorem}
  \label{thm:uniform}
  Let the notation be as in Theorem~\ref{thm:main}. Then given any
  compact set $\cK\subset G/\Gamma$, a bounded continuous function $f$
  on $G/\Gamma$ and an $\epsilon>0$, there exists $t_0>0$ such that
  for any $x\in\cK$ and $t\geq t_0$, we have
  \begin{gather}
    \label{eq:52}
    \Abs{\int_a^b f(a_tu(\phi(s))x)\,ds-\int f\,d\mu_G}\leq \abs{b-a}\epsilon\\
    \Abs{\int_a^b f(a'_tu'(\phi(s))x)\,ds-\int f\,d\mu_G}\leq \abs{b-a}\epsilon.
  \end{gather}
\end{theorem}

The following result is a general uniform version for
Theorem~\ref{thm:action}.

\begin{theorem}
  \label{thm:uniform:action}
  Let $\phi:[a,b]\to \R^{n-1}$ be an analytic curve whose image is not
  contained in a proper affine subspace of $\R^{n-1}$.  Let $L$ be a
  Lie group, $\Lambda$ a lattice in $L$ and $\pi:L\to L/\Lambda$ the
  quotient map. Let $\rho:G\to L$ be a continuous homomorphism. Let
  $\cK$ be a compact subset of $L/\Lambda$. Then given $\epsilon>0$
  and a bounded continuous function $f$ on $L/\Lambda$, there exist
  finitely many proper closed subgroups $H_1,\dots,H_r$ of $L$ such
  that for each $1\leq i\leq r$, $H_i\cap\Lambda$ is a lattice in
  $H_i$ and there exists a compact set
  \begin{equation}
\label{eq:59}
C_i\subset N(H_i,\rho(G)):=\{g\in L:\rho(G)g\subset gH_i\}    
  \end{equation}
  such that the following holds: Given any compact set 
\[
F\subset \cK\setminus \cup_{i=1}^r \pi(C_i)
\]
there exists $t_0>0$ such that for any $x\in F$ and any $t\geq t_0$,.
\begin{equation}
  \label{eq:60}
  \Abs{\frac{1}{b-a}\int_a^b f(\rho(a_tu(\phi(s)))x)\,ds - 
\int_{L/\Lambda} f\,d\mu_L} < \epsilon.
\end{equation}
\end{theorem}

If $L$ is an algebraic group then the sets $N(H_i,\rho(G))$ are
algebraic subvarieties of $L$. Therefore unless $H$ contains a normal
subgroup of $L$ containing $G$, the set $\cup_{i=1}^r \pi(C_i)$ is
contained in a finite union of relatively compact lower dimensional
submanifolds of $\lml$.

From the above result the following statement for translates of any
curve can be deduced using \cite[Lemma~5.1]{Shah:son1}.

\begin{theorem}
  \label{thm:curve:uniform}
  Let $\vpsi:[a,b]\to \SL(n,\R)$ be an analytic curve such that the image
  of the first row of this curve is not contained in a proper subspace
  of $\R^n$.  Let the notation be as in the statement of
  Theorem~\ref{thm:uniform:action}. Then given any compact set
  $F\subset \cK\setminus \cup_{i=1}^r \pi(C_i)$ there exists $t_0>0$
  such that for any $x\in F$ and $t\geq t_0$, we have
\begin{equation}
  \label{eq:64}
  \Abs{\frac{1}{b-a}\int_a^b f(\rho(a_t\vpsi(s))x)\,ds - 
    \int_{L/\Lambda} f\,d\mu_L} < \epsilon.
\end{equation}
\end{theorem}

In the next section we will deduce the number theoretic consequences
from the equidistribution statement. Rest of the article closely
follows the strategy laid out in \S\ref{sec:sketch}. The basic lemma
and its consequences are proved in \S\ref{sec:basic}.
  
\subsubsection*{Acknowledgement} {\small I am very greatful to Shahar
Mozes and Elon Lindenstrauss for the extensive discussions which have
significantly contributed to the process of obtaining the `Basic
Lemma'. Thanks are due to D. Kleinbock, B. Weiss, Y. Bugeaud and
W. Schmidt for their remarks on number theoretic implications of the
results considered here. I would like to thank my wife for her
encourgement.}

\section{Deduction of Theorem~\ref{thm:Dirichlet-2} from
  Theorem~\ref{thm:action}}
The argument given below is based on
\cite[\S2.1]{Kleinbock+Weiss:Dirichlet}.  Let the homomorphism
$\rho(g)=(g,\sigma(g))$ from $G$ to $L=G\times G$, and other notation
be as in \S\ref{sec:Dir2}. Let $n=k+1$. Let $(q_1,\dots,q_k)\in\Z^k$,
$p\in\Z$, $q\in \Z$ and $(p_1,\dots,p_k)\in\Z^k$. Let $N\in\N$ and put
$t=\log(N)$. We consider the standard action of $G\times G$ on
$\R^{n}\times \R^{n}$. For $s\in [a,b]$, we have
\begin{align}
\label{eq:53}
  (\vzeta(N,s),\veta(N,s))&:= \rho(a_tu(\phi(s)))
  \left(
    \left[\begin{smallmatrix} 
        p \\ 
        q_1\\
        \vdots\\
        q_k
      \end{smallmatrix}
    \right],
    \left[\begin{smallmatrix}
        p_k\\
        \vdots\\
        p_1\\
        q
      \end{smallmatrix}
    \right]
  \right)
\\
  &=
  \left(
    \left[\begin{smallmatrix}
        N^k(p+\sum_{i=1}^k q_i\phi_i(s))\\
        N\inv q_1\\
        \vdots\\
        N\inv q_k
      \end{smallmatrix}
    \right],
    \left[\begin{smallmatrix}
        N(q\phi_1(s)+p_k)\\
        \vdots\\
        N(q\phi_k(s)+p_1)\\
        N^{-k}q
      \end{smallmatrix}
    \right]
  \right).
\end{align}
We now fix $0<\mu<1$. Let
\begin{equation}
  \label{eq:55}
  B_\mu=\{(\xi_1,\dots,\xi_n)\in\R^n:\max_{1\leq i\leq n}\abs{\xi_i}\leq\mu\}. 
\end{equation}
Then \eqref{eq:41} is equivalent to $\vzeta(N,s)\in B_\mu$
and \eqref{eq:51} is equivalent to $\veta(N,s)\in B_\mu$. 

Let $\Omega$ denote the space of unimodular lattices in $\R^n$. 
Note that $G$ acts transitively on $\Omega$ and the stabilizer of
$\Z^n$ is $\Gamma$. Similarly $L$ acts transitively on
$\Omega\times\Omega$ and the stabilizer of $x_0:=(\Z^n,\Z^n)$ is
$\Lambda$. Thus $G/\Gamma\cong \Omega$ and $L/\Lambda\cong
\Omega\times\Omega$.

Let
\begin{equation}
  \label{eq:56}
  K_\mu=\{\Delta\in\Omega: \Delta\cap B_\mu=\{\vzero\}\}.
\end{equation}
As we observed above for any $s\in [a,b]$ and $N\in\N$, the
inequalities \eqref{eq:41} and \eqref{eq:51} are simultaneously
insoluble, if for $t=\log N$,
\begin{equation}
  \label{eq:104}
  \rho(a_{t}u(\phi(s)))x_0\in K_\mu\times K_\mu.
\end{equation}

As noted earlier there exists a $\rho(G)$-invariant probability
measure on $\rho(G)x_0\subset \Omega\times\Omega$, say
$\lambda$. Since $\mu<1$, $K_\mu\times K_\mu$ contains an open
neighbourhood of $x_0$. Hence there exists $\epsilon>0$ such that
\begin{equation}
  \label{eq:57}
\lambda(K_\mu\times K_\mu)> \epsilon.
  \end{equation}
  Therefore there exists a continuous function $0\leq f\leq 1$ such that
  \begin{equation}
    \label{eq:40}
    \supp(f)\subset K_\mu\times K_\mu\quad\text{and}\quad 
    \int f\,d\lambda>\epsilon/2.  
  \end{equation}

  Let $J$ be any subinterval of $[a,b]$ with nonempty interior. Then
  by Theorem~\ref{thm:action}, there exists $N_0>0$, such that for all
  $N\geq N_0$, and $t=\log N$,  
  \begin{equation}
    \label{eq:105}
    \frac{1}{\abs{J}}\int_{s\in J} f(\rho(a_{t}u(\phi(s)))x_0)\,ds\geq \epsilon/4,
  \end{equation}
  where $\abs{\cdot}$ denotes the Lebesgue measure on $\R$.  

  Let $\cN\subset\cN$ be an infinite set. Let
\begin{equation}
  \label{eq:106}
  E=\{s\in [a,b]:\rho(a_{(\log N)}u(\phi(s)))x_0\not\in
  K_\mu\times K_\mu \text{ for all large $N\in\cN$}\}.   
\end{equation}
By \eqref{eq:105}, for any subinterval $J\subset[a,b]$, we have
$\abs{J\cap E}\leq (1-\epsilon/4)\abs{J}$.  Therefore $\abs{E}=0$. In
view of the observation associated to \eqref{eq:104}, this proves
Theorem~\ref{thm:Dirichlet-2}. \qed

The Theorem~\ref{thm:Dirichlet} is a special case of
Theorem~\ref{thm:Dirichlet-2}. On the other hand, the proof of
Theorem~\ref{thm:Dirichlet} can also be deduced directly from
Theorem~\ref{thm:main} in a similar way. \qed

\section{Non-divergence of the limiting distribution}
\label{sec:nondiv}
Let $\phi:I=[a,b]\to \R^{n-1}$ be a nonconstant analytic map. Let $x_0\in
\lml$. Given a nontrivial continuous homomorphism $\rho:G\to L$, for the
sake of simplicity of notation in our study, without loss of
generality, we will identify $g\in G$ with $\rho(g)\in L$. Therefore
now onwards we will treat $G$ as a subgroup of $L$, where $\rho$ being
an inclusion. Let $t_i\to\infty$ be any sequence in $\R$. Let $x_i\to
x_0$ be a convergent sequence in $\lml$. For any $i\in\N$ let $\mu_i$
be the measure on $\lml$ defined by
\begin{equation}
  \label{eq:13}
  \int_{\lml} f\,d\mu_i:=\frac{1}{\abs{I}}\int_I f(a_{t_i}u(\phi(s))x_i)\,ds,
  \quad\forall f\in\Cc(\lml).
\end{equation}

\begin{theorem}
  \label{thm:return}
  Given $\epsilon>0$ there exists a compact set $\cF\subset \lml$ such
  that $\mu_i(\cF)\geq 1-\epsilon$ for all large $i\in\N$.
\end{theorem}

It may be noted that in the case of $L=G$, $\rho$ the identity map
and $\Lambda=\SL(n,\Z)$, the above result was already proved by
Kleinbock and Margulis~\cite{Klein+Mar:Annals98}. The rest of this
section is devoted to obtaining the same conclusion in the case of
arbitrary $L$ and $\Lambda$. 

\subsection{}
Let $\sH$ denote the collection of analytic subgroups $H$ of $G$ such
that $H\cap\Lambda$ is a lattice in $H$, and a unipotent one-parameter
subgroup of $H$ acts ergodically with respect to the $H$-invariant
probability measure on $H/H\cap\Lambda$.  Then $\sH$ is a countable
collection \cite{Shah:uniform,R:measure}.

Let $\la{L}$ denote the Lie algebra associated to $L$. Let
$V=\oplus_{d=1}^{\dim\la{L}}\wedge^d\la{L}$ and consider the
$(\oplus_{d=1}^{\dim{\la{l}}}\wedge^d\Ad)$-action of $L$ on $V$. Given
$H\in\sH$, let $\la{h}$ denote its Lie algebra, and fix
$p_H\in\wedge^{\dim\la{h}}\la{h}\setminus \{0\}\subset V$. Then
\begin{equation}
  \label{eq:27}
  \Stab_L(p_H)=\noroneL(H):=\{g\in \norL(H):\det((\Ad g)|_{\la{H}})=1\}.
\end{equation}

\begin{proposition}[\cite{Dani+Mar:limit}]
  \label{prop:discrete}
  The orbit $\Lambda \cdot p_H$ is a discrete subset of $V$.  \qed
\end{proposition}

We may note that when $L$ is a real algebraic group defined over $\Q$
and $\Lambda=L(\Z)$ then the above countability result and the
discreteness of the orbit are straightforward to prove.

\subsection{Functions with growth factor $C$ and growth order
  $\alpha$}
\label{subsec:CAlpha-1}
Let $\sF$ denote the $\R$-span of all the the coordinate functions of
the map $\Upsilon:I\to \End(V)$ given by
$\Upsilon(s)=(\oplus_{d=1}^{\dim{\la{l}}}\wedge^d\Ad)(u(\phi(s)))$ for
all $s\in \R$.  As explained in \cite[\S2.1]{Shah:son1}, due to
\cite[Proposition~3.4]{Klein+Mar:Annals98} the family $\sF$ has the
following growth property for some $C>0$ and $\alpha>0$: for any
subinterval $J\subset I$, $\xi\in\sF$ and $r>0$,
\begin{equation} 
  \label{eq:CAlpha2} 
  \abs{\{s\in J: \abs{\xi(s)}<r\}} <
  C\left(\frac{r}{\sup_{s\in J}\abs{\xi(s)}}\right)^\alpha\abs{J}.
\end{equation}

As a direct consequence of this condition, we have the following
\cite{Klein+Mar:Annals98}: Fix any norm $\norm{\cdot}$ on $V$.

\begin{proposition} 
\label{prop:relative} 
Given any $\epsilon>0$ and $r>0$, there exists $R>0$ such that for any
$h_1,h_2\in L$ and a subinterval $J\subset I$, one of the following
holds:
\begin{enumerate}
\item[I)] $\sup_{t\in J}\norm{h_1u(\phi(t))h_2p_H}<R$.
\item[II)] 
\begin{align*}
\abs{\{t\in J:  \norm{h_1u(\phi(t))h_2 p_H}&\leq r\}}\\
               &\leq \epsilon 
   \abs{\{t\in J:\norm{h_1u(\phi(t))h_2 p_H}\leq R\}}.
\end{align*}
\end{enumerate}
\qed
\end{proposition}

\subsection{The non-diverergence criterion} 

\begin{proposition}
\label{prop:return}
There exist closed subgroups $W_1,\ldots,W_r\in\sH$ (depending only on
$L$ and $\Lambda$) such that the following holds: Given $\epsilon>0$
and $R>0$, there exists a compact set $\cF \subset L/\Lambda$ such
that for any $h_1,h_2\in L$ and a subinterval $J\subset I$, one of
the following conditions is satisfied:
\begin{enumerate}
\item[I)]
There exists $\gamma \in \Lambda$ and $i \in \{1, \ldots,r \}$ such that
\[
\sup_{s \in J} \norm{h_1u(\phi(s))h_2 p_H}< R.\]
\item[II)] 
$\frac{1}{\abs{J}}\abs{\{s \in J : (h_1u(\phi(s))h_2)\Lambda/\Lambda \in
  \cF \}} \geq 1- \epsilon$.
\end{enumerate}
\end{proposition}

\begin{proof}
  The result follows from the argument as in
  \cite[Theorem~2.2]{Shah:horo} using the earlier results of Dani and
  Margulis~\cite{DM:asymptotic}; as well as its extensions due to
  Kleinbock and Margulis~\cite{Klein+Mar:Annals98}. The main
  difference is that instead of growth properties of polynomial
  functions, one uses the similar properties of functions in $\sF$ as
  given by Proposition~\ref{prop:relative}.
\end{proof}

\begin{proof}[Proof of Theorem~\ref{thm:return}]
  Take any $\epsilon>0$. Take a sequence $R_k\to 0$ as
  $k\to\infty$. For each $k\in\N$, let $\cF_k\subset\lml$ be a compact
  set as determined by Proposition~\ref{prop:return} for these
  $\epsilon$ and $R_k$. If the theorem fails to hold, then for each
  $k\in\N$ we have $\mu_i(\cF_k)<1-\epsilon$ for all large $i$.
  Therefore after passing to a subsequences of $\{\mu_i\}$, we may
  assume that $\mu_i(\cF_i)<1-\epsilon$ for all $i$. Then by
  Proposition~\ref{prop:return}, after passing to a subsequence, we
  may assume that there exists $W\in\sH$ such that for each $i$ there
  exists $\gamma_i\in\Lambda$ such that
\begin{equation}
  \label{eq:65}
  \norm{\sup_{s\in I}a_{t_i}u(\phi(s))\gamma_ip_W}\leq
  R_i\toinfty 0.
\end{equation}
By Proposition~\ref{prop:discrete}, there exists $r_0>0$ such that
$\norm{\gamma_ip_W}\geq r_0$ for each $i$. We put
$v_i=\gamma_ip_W/\norm{\gamma_i p_W}$. Then $v_i\to v\in V$ and
$\norm{v}=1$. Therefore
\begin{equation}
  \label{eq:66}
  \sup_{s\in I}\norm{a_{t_i}u(\phi(s))v_i}\leq R_i/r_0\to 0 \quad
  \text{as $i\to\infty$}.
\end{equation}

Define
\begin{align}
  \label{eq:68}
  V^-&=\{w\in V:\lim_{t\to\infty} a_tw=0\}\\
  V^+&=\{w\in V:\lim_{t\to\infty} (a_t)\inv w=0\}\\
  V^0&=\{w\in V: a_tw=w, \quad\forall t\in\R\}.
\end{align}
Since $\{a_t\}$ acts on $V$ via commuting $\R$-diagonalizable
matrices, we have that $V=V^+\oplus V^0\oplus V^-$. Let $\pi_0:V\to
V^0$ denote the associated projection.  Then from \eqref{eq:66} we
conclude that
  \begin{equation}
    \label{eq:67}
    u(\phi(s))v\subset V^-, \quad\forall s\in I.
  \end{equation}
  
  The `Basic Lemma' (Proposition~\ref{prop:main}) proved in the next
  section states that for any finite dimensional linear representation
  $V$ of $G$, any $v\in V\setminus\{0\}$ and any $\cB\subset \R^{n-1}$
  which is not contained in a proper affine subspace, if
\begin{equation}
  \label{eq:95}
  u(\phi(e))v\in V^-+V^0, \quad \forall e\in\cB,
\end{equation}
then
\begin{equation}
  \label{eq:96}
  \pi_0(u(\phi(e))v)\neq 0,\quad \forall e\in\cB.
\end{equation}

By our hypothesis \eqref{eq:67} implies \eqref{eq:95} but contradicts
its consequence \eqref{eq:96}. 
\end{proof}

As a consequence of Theorem~\ref{thm:return} we deduce the following:

\begin{corollary} \label{cor:mu-return} After passing to a
  subsequence, $\mu_i\to\mu$ in the space of probability measures on
  $\lml$ with respect to the weak-$\ast$ topology.
\end{corollary}

Before we proceed further from here, we will give a proof of the Basic
lemma and obtain its consequence, which will be used in the later
sections.

\section{Dynamics of linear actions of intertwined $\SL(2,\R)$'s}
\label{sec:basic}

A triple $(X,H,Y)$ of elements of a Lie algebra is called an
\emph{$\la{SL}_2$-triple\/} if
\begin{equation}
  \label{eq:2}
  [X,Y]=H, \quad [H,X]=2X, \quad [H,Y]=-2Y.
\end{equation}

The following observation on linear dynamics of $\SL_2$-action
played a crucial role in understanding limiting distributions of
expanding translates of curves under the geodesic flows on hyperbolic
manifolds \cite{Shah:son1}.

\newcommand{\citone}{{\cite[Lemma~2.3]{Shah:son1}}}
\begin{lemma}[\citone]
  \label{lema:sl2}
  Let $W$ be a finite dimensional irreducible representation of an
  $\la{SL}_2$-triple $(X,H,Y)$. Let $W^-$ (respectively, $W^+$) denote
  the sum of strictly negative (respectively, strictly positive)
  eigenspaces of $H$. Let $\pi_{+}:W\to W^+$ denote the projection
  parallel to the eigenspaces of $H$. Then 
  \begin{equation}
    v\in W^-\setminus\{0\}\implies \pi_+(\exp(X)v)\neq 0.
  \end{equation}
  \qed
\end{lemma}

The main goal of this section is to obtain a similar result on linear
dynamics of $\SL(n,\R)$-actions by considering intertwined actions of
copies of $\SL(2,\R)$'s contained in $\SL(n,\R)$.

\subsubsection{Notation}
Let $\cA=\diag((n-1),-1,\dots,-1)\in\la{sl}(n,\R)$. Then
$a_t=\exp(t\cA)$ for all $t\in\R$. Define $A=\{a_t:t\in\R\}$ and
$\la{a}=\Lie(A)=\R\cdot \cA$. Consider a linear representation of $G$
on a finite dimensional vector space $V$.  For $\mu\in\R$, define
  \begin{equation}
    \label{eq:1}
    V^\mu =\{v\in V:\cA v=\mu v\}.
  \end{equation}
  Let $\pi_\mu:V\to V^{\mu}$ denote the projection parallel to the
  eigen spaces of $\cA$. Put
  \begin{align}
    \label{eq:Vpm}
    V^-&= \sum_{\mu<0} V^\mu, \quad V^+=\sum_{\mu>0} V^\mu\\
    \pi_-=&\sum_{\mu<0} \pi_{\mu},\quad \pi_+=\sum_{\mu>0} \pi_{\mu}.
  \end{align}

  An {\em affine basis\/} of $\R^{n-1}$ is a set $\cB\subset\R^{n-1}$
  such that for any $e\in\cB$, the set
  $\{e'-e:e'\in\cB\setminus\{e\}\}$ is a basis of $\R^{n-1}$.

  \begin{proposition}[Basic Lemma]
    \label{prop:main}
    Let the notation be as above. Given an affine basis $\cB$ of
    $\R^{n-1}$ and a nonzero vector $v\in V$, suppose that
    \begin{equation}
      \label{eq:69}
      u(e)v\in V^0+V^-, \quad \forall e\in\cB.
    \end{equation}
    Then
    \begin{equation}
      \label{eq:70}
      \pi_0(u(e)v)\neq 0,\quad \forall e\in\cB.
    \end{equation}
  \end{proposition}
  
\begin{proof}
  By \eqref{eq:69} there exists $\mu_0\leq 0$ and $e_0\in\cB$ such
  \begin{align}
    \label{eq:mu0}
    \pi_{\mu_0}(u(e_0)v)&\neq 0, \quad \text{and}\\
    \label{eq:7}
    \pi_{\mu}(u(e)v)&=0,\quad \text{$\forall \mu>\mu_0$ and $\forall
      e\in\cB$}.
  \end{align}

  We write the basis $\{e-e_0:e\in\cB\setminus\{e_0\}\}$ 
  $\R^{n-1}$ as $\{e_1,\dots,e_{n-1}\}$. Put $v_0=u(e_0)v$. Then
  \begin{align}
    \label{eq:3}
    \pi_{\mu_0}(v_0)&\neq 0\quad\text{and} \\
    \label{eq:3b}
    \pi_{\mu}(u(e_i)v_0)&=0\quad \text{for all $\mu>\mu_0$ and
      $1\leq i\leq n-1$}.
  \end{align}
  To prove \eqref{eq:70} we need to show that
  \begin{align}
    \label{eq:109}
    \mu_0&=0 \quad \text{and} \\
    \label{eq:109b}
    \pi_0(u(e_i)v_0)&\neq 0, \quad\forall 1\leq i\leq n-1.
  \end{align}

  Let the set $\{f_1,\dots,f_{n-1}\}$ consisting of real $((n-1)\times
  1)$ column matrices be the dual to the basis $\{e_1,\dots,e_{n-1}\}$
  of $\R^{n-1}$ consisting of $(1\times (n-1))$-row matrices; that is,
\begin{equation}
  \label{eq:4}
  e_if_j =\delta_{i,j}, \quad\forall\, i,j\in \{1,\dots,n-1\},
\end{equation}
For $i\in\{1,\dots,n-1\}$, 
let 
\begin{equation}
  \label{eq:97}
 X_i:=X(e_i)=
\left[
\begin{smallmatrix} 
0 & e_i \\
  & \vzero_{n-1}
\end{smallmatrix}
\right] \quad \text{and} \quad
Y_i:=Y(f_i)=
\left[
  \begin{smallmatrix}
    0  & \\
    f_i&\vzero_{n-1}
  \end{smallmatrix}
\right],
\end{equation}
where $\vzero_{n-1}$ is the $((n-1)\times(n-1))$-zero matrix. Then
$u(e_i)=\exp(X_i)$. Let
\begin{equation}
  \label{eq:5}
  H_i:=[X_i,Y_i]= 
  \left[ 
    \begin{smallmatrix}
      e_if_i & \\
      & - f_ie_i
    \end{smallmatrix}
  \right]
  =
  \left[
    \begin{smallmatrix}
      1 & \\
      & -f_ie_i
    \end{smallmatrix}
  \right]\in\la{sl}(n,\R).
\end{equation}
Then $(X_i,H_i,Y_i)$ is an $\la{SL}_2$-triple. Let
$\la{g}_i=\Span\{X_i,H_i,Y_i\}\subset\la{sl}(n,\R)$. By \eqref{eq:4}
\begin{equation}
  \label{eq:90}
  \sum_{i=1}^{n-1}
  f_ie_i = I_{n-1},  
\end{equation}
the $((n-1)\times (n-1))$-identity matrix. Therefore
\begin{equation}
  \label{eq:6}
  \cA=H_1+\dots+H_{n-1}.
\end{equation}
Also
\begin{equation}
  \label{eq:72}
  \la{b}=\Span\{H_i:i=1,\dots,n-1\}
\end{equation}
is a maximal $\R$-diagonalizable subalgebra of
$\la{SL}(n,\R)$. Moreover
\begin{equation}
  \label{eq:73}
  [H,\la{g_i}]=\la{g_i},\quad \forall H\in\la{b}.
\end{equation}
Thus $\la{b}+\la{g_i}$ is a reductive Lie algebra which is isomorphic
to $\R^{n-2}\oplus \la{SL}_2$. Note that the Lie groups associated to
these $\la{g}_i$'s are our intertwined copies of $\SL_2$'s, and we
want to study their joint linear dynamics.

For a linear functional $\delta\in\la{b}^\ast$, let
\begin{equation}
  \label{eq:91}
  V(\delta)=\{v\in V:Hv=\delta(H)v\}.
\end{equation}
The set $\Delta=\{\delta\in\la{b}^\ast:V(\delta)\neq 0\}$ is
finite and $V=\oplus_{\delta\in\Delta} V(\delta)$. Let
$q_\vlambda:V\to V(\delta)$ be the associated projection. 
  
\begin{claim} \label{claim:nonneg} Let $\delta\in\Delta$ be such that
  $v(\delta):=q_\delta(\pi_{\mu_0}(v_0))\neq 0$.  Then
  $\delta(H_i)\geq 0$ for all $1\leq i\leq n-1$.
  \end{claim}

  To prove the claim, take any $1\leq i\leq n-1$.  Consider the
  decomposition
  \begin{equation}
    \label{eq:12}
    V=W_1\oplus \dots \oplus W_s,
  \end{equation}
  where $W_j$'s are irreducible subspaces for the action of the Lie
  subalgebra $\la{b}+\la{g}_i$ and $s\in\N$. Therefore each $W_j$ is
  an irreducible representation of the $\la{SL}_2$-triple
  $(X_i,H_i,Y_i)$. Let $P_j:V\to W_j$ denote the associated
  projection. We note that
  \begin{align}
    \label{eq:13a}
    \pi_\mu\circ P_j & = P_j\circ \pi_\mu,\quad \text{for all $1\leq
      j\leq s$ and $\mu\in\R$, and } \\
    \label{eq:13b}
    q_\delta\circ P_j &= P_j\circ q_\delta, \quad \text{for all $1\leq
      j\leq s$.} 
  \end{align}

  There exists $1\leq j\leq s$ such that
  \begin{equation}
    \label{eq:24}
    P_j(v(\delta))\neq 0,
  \end{equation}
  we take any such $j$. In particular, by \eqref{eq:13a} and
  \eqref{eq:13b},
  \begin{equation}
    \label{eq:18}
    W_j\cap V_{\mu_0}\ni P_j(\pi_{\mu_0}(v_0))\neq 0.
  \end{equation}
  By the standard description of finite dimensional representations of
  $\la{SL}_2$, let $k\geq 0$ and $w_{-k}\in W_j$ be such that
  \begin{equation}
    \label{eq:14}
    Y_i\cdot w_{-k}=0 \quad \text{and}\quad H_i\cdot w_{-k}=-k\cdot w_{-k}.  
  \end{equation}
  For any $r\geq 0$, put $w_{-k+2r}:=X_i^r\cdot w_{-k}$. Then
  \begin{equation}
    \label{eq:16}
    H_i\cdot w_{-k+2r}=(-k+2r)w_{-k+2r}
  \end{equation}
  and $W_j=\Span\{w_{-k},\dots,w_k\}$. Since $[H_i,\la{b}]=0$ and
  $W_j$ is $\la{b}$-invariant, for each $0\leq r\leq k$, there exists
  $\delta_r\in \Delta$ such that $w_{-k+2r}\in V(\delta_r)$ and
\begin{equation} 
\label{eq:lambda_r}  
\delta_r\neq \delta_{r'},\quad\text{if $r\neq r'$}. 
\end{equation}  

Put $\lambda=\delta_0(\cA)$. Then
  \begin{equation}
    \label{eq:15}
    \cA\cdot w_{-k}=\lambda\cdot w_{-k}.
  \end{equation}
Since $[\cA,X_i]=n$, we have
\begin{equation}
  \cA\cdot w_{-k+2r}=(\lambda+nr) w_{-k+2r},\quad\forall 0\leq r\leq k.
\end{equation}
Thus, if $P_j(V_\mu)\neq 0$ for any $\mu$, then $P_j(V_\mu)\subset
\R\cdot w_{-k+2r}$ for some $r\geq 0$ such that
$\lambda+nr=\mu$.

Therefore by \eqref{eq:18} there exists $r_0\geq 0$ such that
\begin{align}
  \label{eq:19}
  \mu_0&=\lambda+nr_0 \quad \text{and} \\
    \label{eq:10}
   W_j\cap V_{\mu_0}&=\R\cdot w_{-k+2r_0}.
  \end{align}
  Recall that $u(e_i)=\exp(X_i)$. By \eqref{eq:3b} and \eqref{eq:13a},
  for all $\mu>\mu_0$, we have
  \begin{align}
    \label{eq:20}
    \pi_\mu(P_j(v_0))&=P_j(\pi_\mu(v_0))=0,\quad \text{and}\\
    \pi_\mu(\exp(X_i)P_j(v_0))&=\pi_\mu(P_j(\exp(X_i)v_0))=
    P_j(\pi_\mu(\exp(X_i)v_0))=0.
  \end{align}
  Therefore
  \begin{equation}
    \label{eq:21}
    \{P_j(v_0),\exp(X_i)P_j(v_0)\}\subset \Span\{w_{-k},\dots,w_{-k+2r_0}\}.
  \end{equation}
  Therefore by Lemma~\ref{lema:sl2} applied to the $\la{sl}_2$-triple
  $(X_i,H_i,Y_i)$, since $P_j(v_0)\neq 0$, we have
  \begin{equation}
    \label{eq:22}
    -k+2r_0\geq 0.  
  \end{equation}
  
  By \eqref{eq:24}, \eqref{eq:lambda_r} and \eqref{eq:10},
  \begin{equation}
    \label{eq:36}
    0\neq P_j(v(\delta)) = P_j(q_\delta(\pi_{\mu_0}(v_0))) = 
q_\delta(P_j(\pi_{\mu_0}(v_0)))=P_j(\pi_{\mu_0}(v_0)).
  \end{equation}
Also
\begin{equation}
  \label{eq:63}
  H_i(P_j(v(\delta)))=P_j(H_i(v(\delta))) = P_j(\delta(H_i)v(\delta)) = 
  \delta(H_i)\cdot P_j(v(\delta)),
\end{equation}
and 
\begin{equation}
  \label{eq:108}
H_i(P_j(\pi_{\mu_0}(v_0)))=(-k+2r_0)P_j(\pi_{\mu_0}(v_0)).  
\end{equation}
Therefore by \eqref{eq:22}
  \begin{equation}
    \label{eq:25}
    \delta(H_i)=-k+2r_0\geq 0.   
  \end{equation}  
This completes the proof of the Claim~\ref{claim:nonneg}.

Since $\pi_{\mu_0}(v_0)\neq 0$, there exists $\delta\in\Delta$ such
that 
\[
v(\delta):=q_\delta(\pi_{\mu_0}(v_0))\neq 0.
\]
Now since $\cA\cdot v(\delta)=\mu_0 v(\delta)$ and
$\cA=H_1+\dots+H_{n-1}$, by Claim~\ref{claim:nonneg}
  \begin{equation}
    \label{eq:71}
    0\geq \mu_0=\sum_{i=1}^{n-1} \delta(H_i)\geq 0.
  \end{equation}
  Therefore
  \begin{equation}
    \label{eq:111}
    \mu_0=0 \quad\text{and}\quad \delta(H_i)=0, \quad\forall 1\leq
    i\leq n-1.
  \end{equation}
  Thus \eqref{eq:109} is verified. 

  Going back to the representation $W_j$ of $\la{g}_i$ considered
  above, by \eqref{eq:25} and \eqref{eq:111} $-k+2r_0=\delta(H_i)=0$
  and $\cA\cdot w_0=0$. Therefore by \eqref{eq:21}, we have
  \begin{equation}
    \label{eq:75}
    \{P_j(v_0),\exp(X_i)P_j(v_0)\}\subset \Span\{w_{-k},\dots,w_{0}\}.
  \end{equation}
  Therefore, since $P_j(v_0)\neq 0$, by Lemma~\ref{lema:sl2}
  \begin{equation}
    \label{eq:76}
    \exp(X_i)P_j(v_0)\not\subset \Span\{w_{-k},\dots,w_{-2}\}.
  \end{equation}
  Hence by \eqref{eq:10} and \eqref{eq:13a},
  \begin{equation}
    \label{eq:77}
    P_j(\pi_0(\exp(X_i)v_0)) = \pi_0(P_j(\exp(X_i)v_0)) =
    \pi_0(\exp(X_i)P_j(v_0))\neq 0.
  \end{equation}
  Therefore $\pi_0(\exp(X_i)v_0)\neq 0$. Thus \eqref{eq:109b} is
  verified.
\end{proof}

Consider the linear action of $\cen(A)$ on $\R^{n-1}$ such that
\begin{equation}
  \label{eq:80}
  u(g\cdot e)=gu(e)g\inv, \forall g\in \cen(A),\ \forall e\in\R^{n-1}.  
\end{equation}
Note that under this action $\cen(A)$ maps onto $\GL(n-1,\R)$.  We
also note that for any basis $\cC$ of $\R^{n-1}$, the set
\begin{equation}
  D_{\cC}:=\{g\in \cen(A):\text{each $e\in\cC$ is an eigenvector of $g$}\}
\end{equation}
is a maximal $\R$-diagonalizable subgroup of $G$.

\begin{corollary}
  \label{cor:main2}
  Let the notation be as in Proposition~\ref{prop:main}. For any
  $e\in \cB$,
  \begin{equation}
    \label{eq:112}
    g\pi_0(u(e)v)=\pi_0(u(e)v), \quad\forall g\in D_{\cC},
  \end{equation}
  where $\cC=\{e'-e:e'\in\cB\setminus\{e\}\}$.
\end{corollary}

\begin{proof}
  As a consequence of Proposition~\ref{prop:main}, $\mu_0=0$ and
  $\pi_{\mu_0}(u(e)v)\neq 0$. Let the notation be as in the proof of
  Proposition~\ref{prop:main}. Then $u(e)v=v_0$. Now
  $\cC=\{e_1,\dots,e_{n-1}\}$ and for any $i$, we have $H_ie_i=2$ and
  $H_ie_j=e_j$ if $j\neq i$. Therefore by \eqref{eq:72},
  $D_{\cC}=\exp(\la{b})$.

  For any $\delta\in\Delta$, if $q_\delta(\pi_{\mu_0}(v_0))\neq 0$
  then by \eqref{eq:111}
  \begin{equation}
    \label{eq:92}
    H_iq_\delta(\pi_{\mu_0}(v_0))=\delta(H_i)q_\delta(\pi_{\mu_0}(v_0))=0, 
    \quad \forall 1\leq i\leq n-1.
  \end{equation}
  Therefore $H_i\pi_{\mu_0}(v_0)=\sum_{\delta\in\Delta}
  H_iq_\delta(\pi_{\mu_0}(v_0))=0$ for all $i$.  Hence
  $\la{b}\cdot\pi_0(v_0)=0$.  Therefore
  $D_\cC\pi_0(v_0)=\pi_0(v_0)=p_0(u(e)v)$ and we obtain
  \eqref{eq:112}.
\end{proof}

\begin{corollary}
  \label{cor:curve}
  Let a set $\cE\subset \R^{n-1}$ and $e\in\cE$ be such that the set
  $\cE_e:=\{e'-e:e'\in\cE\}$ is not contained a union of $n-1$ proper
  subspaces of $\R^{n-1}$. Suppose that $v\in\R^{n-1}$ is such that
  \begin{equation}
    \label{eq:79}
    u(e')v\in V^0+V^-,\quad \forall e'\in\cE.
  \end{equation}
  Then
  \begin{gather}
    \label{eq:11b}
    \pi_0(u(e)v)\neq 0 \quad\text{and}\\
    \label{eq:11}
    \cen(A)\subset \Stab_G(\pi_0(u(e)v)).
  \end{gather}
\end{corollary}

\begin{proof}
  First \eqref{eq:11b} follows from Proposition~\ref{prop:main}.

  Replacing $v$ by $u(e)v$ and $\cE$ by $\cE_e$, without loss of
  generality we may assume that $e=0\in\cE$ and we only need to prove
  that
  \begin{equation}
    \label{eq:89}
    \cen(A)\subset \Stab_G(\pi_0(v)).    
  \end{equation}
  
  By our hypothesis there exist $\{b_1,\dots,b_{n-1}\}\subset\cE$
  which is a basis of $\R^{n-1}$.  Let $\{e_i:i=1,\dots,n-1\}$ denote
  the standard basis of $\R^{n-1}$. We put $e_0=0$. Then there exists
  $z\in \cen(A)$ such that $zb_i=e_i$ for $1\leq i\leq n-1$. Now by
  \eqref{eq:79},
  \begin{equation}
    \label{eq:81}
    z(u(b)v)=(zu(b)z\inv)(zv)=u(z\cdot b)(zv)\subset
    V^0+V^-, \quad \forall b\in\cE.
  \end{equation}
  Also $\pi_0(zw)=z\pi_0(w)$ for all $w\in V$. Therefore to prove the
  result, without loss of generality we can replace $\cE$ with $z\cdot
  \cE$ and $v$ with $zv$, and assume that 
  \begin{equation}
    \label{eq:94}
  \cB:=\{e_i:0\leq i\leq n-1\}\subset \cE.
  \end{equation}
  
  Let $\cC=\{e_1,\dots,e_{n-1}\}$. Let $D$ denote the maximal diagonal
  subgroup of $\SL(n,\R)$. Then $D_{\cC}=D$ and by
  Corollary~\ref{cor:main2},
  \begin{equation}
    \label{eq:83}
    D\subset \Stab(\pi_0(v)).
  \end{equation}

  By our hypothesis, $\cE\setminus\{e_1,\dots,e_{n-1}\}$ is not
  contained in a proper subspace of $\R^{n-1}$. Therefore there exists
  $e_1'\in\cE$ such that
  \begin{equation}
    \label{eq:84}
    e_1'=\sum_{i=1}^{n-1} \lambda_ie_i \quad\text{and}\quad
    \lambda_i\neq 0,\quad\forall\, 1\leq i\leq n-1.
  \end{equation}

  For $\vx=(x_1,\dots,x_{n-2})\in\R^{n-2}$, let 
  \begin{equation}
    \label{eq:8}
w(\vx)=\left[
\begin{smallmatrix}
  1 \\   
    &   1 & x_2 & \dots & x_{n-2} \\
    &     &  1 \\
    &     &     &\ddots \\
    &     &     &        &  1
\end{smallmatrix}\right]
  \end{equation}
  Put
  $\vx=(-\lambda_1\inv\lambda_2,\dots,-\lambda_1\inv\lambda_{n-1})$. Then
  $w(\vx)e_1=\lambda_1\inv e_1'$ and $w(\vx) e_i=e_i$ for
  $i=2,\dots,n-1$. Let $\cC'=\{e_1',e_2,\dots,e_{n-1}\}$. Then
  \begin{equation}
    \label{eq:86}
    D_{\cC'}=w(\vx)\inv D w(\vx).  
  \end{equation}
  
  Since $\cB':=\{0,e_1',e_2,\dots,e_{n-1}\}\subset \cE$, by
  Corollary~\ref{cor:main2} applied to $\cB'$ in place of $\cB$, we
  obtain that
  \begin{equation}
    \label{eq:87}
    w(\vx)\inv D w(\vx)=D_{\cC'} \subset \Stab(\pi_0(v)).
  \end{equation}

  Since each coordinate of $\vx$ is nonzero, the group generated by
  $D$ and $w(\vx)\inv D w(\vx)$ contains
  $W:=\{w(\vy):\vy\in\R^{n-2}\}$. Then by \eqref{eq:83} and
  \eqref{eq:87}, $DW\subset\Stab_G(\pi_0(v))$. Let $\alpha>1$. Then
  $W$ is the expanding horospherical subgroup of $\cen(A)$ associated
  to
  \[
  g_0=\diag(1,\alpha^{n-2},\alpha\inv,\dots,\alpha\inv)\in D.
  \] 
  Note that $W^-:=\{\trn{w(vy)}:\vy\in\R^{n-2}\}$ is the contracting
  horospherical subgroup of $\cen(A)$ associated to $g_0$. Therefore
  $\pi_0(v)$ is stabilized by $W^-$. Since $\cen(A)\cong
  A\cdot\SL(n-1,\R)$, one verifies that $\cen(A)$ is generated by $W$,
  $W^-$ and $D$. Therefore $\cen(A)$ stabilizes $\pi_0(v)$; that is
  \eqref{eq:89} holds.
\end{proof}

\begin{lemma}
  \label{lema:accu}
  Let $x_i\to x$ be a convergent sequence in $\R^{n-1}$. Suppose there
  exists $v\in V$ such that
\begin{equation}
    \label{eq:88}
    u(x_i)v\in V^0+V^-,\quad \forall i\in \N,
  \end{equation}
  and a sequence $\delta_i\to 0$ of nonzero reals such that
  $f=\lim_{i\to\infty}(x_i-x)/\delta_i$ exists. Then
  \begin{equation}
    \label{eq:78}
    u(f)\in\Stab(\pi_0(u(x)v)).
  \end{equation}
\end{lemma}

\begin{proof}
For any sequence $t_i\to\infty$ and $w_i\to w$ in $V^-+V^0$, we have
$a_{t_i}w_i\to \pi_0(w)$ as $i\to\infty$. In particular,
\begin{equation}
  \label{eq:54}
 a_{t_i}u(x_i)v\toinfty\pi_0(u(x)v). 
\end{equation}

Put $t_i=(1/n)\log(\delta_i\inv)$ for all $i$. Then 
\begin{align*}
  a_{t_i}u(x_i)&=a_{t_i}u(x_i-x)u(x)v \\ \
              &=u(e^{nt_i}(x_i-x))(a_{t_i}u(x)v) \toinfty u(f)\pi_0(u(x)v).
\end{align*}
Therefore $u(f)$ stabilizes $\pi_0(u(x)v)$.
\end{proof}

\begin{corollary}
  \label{cor:curve-main}
  Let $\phi:I=[a,b]\to \R^{n-1}$ be a differentiable curve which is
  not contained in a proper affine subspace of $\R^{n-1}$. Let $0\neq
  v\in V$ be such that
  \begin{equation}
    \label{eq:93}
    u(\phi(s))v\in V^0+V^-, \quad \forall s\in I.
  \end{equation}
  Then $v$ is stabilized by $G$.
\end{corollary}

\begin{proof}
  Since $\phi$ is differentiable and $\phi(I)$ not contained in a
  proper affine subspace of $\R^{n-1}$, we conclude that for any
  $s_0\in I$, the set $\cE_0:=\{\phi(s)-\phi(s_0):s\in I\}$ is not
  contained in a finite union of proper subspaces of $\R^{n-1}$. Now
  we can apply Corollary~\ref{cor:curve} to the set
  $\cE=\{\phi(s):s\in I\}$ and conclude that $\pi_0(u(\phi(s))v)\neq
  0$ is stabilized by $\cen(A)$ for all $s\in I$. Now let $s_0\in I$
  such that $\dot\phi(s_0)\neq 0$. Let $\delta_i\to 0$ be a sequence
  of of nonzero reals. Then
  $\dot\phi(s_0)=\lim_{i\to\infty}(\phi(s_i)-\phi(s_0))/\delta_i$. Therefore
  by Lemma~\ref{lema:accu} $\pi_0(u(\phi(s_0))v)$ is stabilized by
  $u(\dot\phi(s_0))$. Now the subgroup, say $Q$, generated by
  $u(\dot\phi(s_0))$ and $\cen(A)$ contains
  $\{u(\vx):\vx\in\R^{n-1}\}$. Therefore $Q$ is a parabolic subgroup
  of $G$. Since $Q$ stabilizes $\pi_0(u(\phi(s_0))v)$, we conclude
  that $G$ stabilizes $\pi_0(u(\phi(s_0))v)$.

  We put $v_0=u(\phi(s_0)v)-\pi_0(u(\phi(s_0)v))$. Then 
  \begin{equation}
    \label{eq:128}
    u(\phi(s)-\phi(s_0))v_0=u(\phi(s))v+\phi_0(u(\phi(s_0))v)\in
    V^-0+V^-,\quad\forall s\in I.
  \end{equation}
  We choose a finite subset $I_1\subset I$ containing $s_0$ such that
  $\{\phi(s)-\phi(s_0):s\in I_1\}$ is an affine basis of $\R^{n-1}$,
  and apply Proposition~\ref{prop:main}. Therefore if $v_0\neq 0$ then
  $\pi_0(v_0)\neq 0$. Since by our choice $\pi_0(v_0)=0$, we conclude
  that $v_0=0$. Therefore $u(\phi(s_0))v$ is stabilized by $G$. Hence
  $v$ is stabilized by $G$, because $u(\phi(s_0))\in G$.
\end{proof}

\begin{corollary}
  \label{cor:expand}
  Let $\phi:[a,b]\to \R^{n-1}$ be a differentiable curve, whose image
  is not contained in a proper affine subspace of $\R^{n-1}$. Let $V$
  be a finite dimensional normed linear space on which $G$ acts
  linearly. Further suppose that there is no nonzero $G$-fixed vector
  in $V$.  Then given $C>0$ there exists $t_0>0$ such that
    \begin{equation}
      \label{eq:31}
      \sup_{s\in I} \norm{{a_t}u(\phi(s))v}\geq C\norm{v},\quad \text{$\forall
      v\in V$ and $t>t_0$}.
    \end{equation}
  \end{corollary}

\begin{proof}
If the conclusion of the proposition fails to hold then there exists
$C>0$ and a sequence $t_i\to\infty$ and convergent sequence $v_i\to v$
in $V$ such that $\norm{v}=1$, and
\begin{equation}
      \label{eq:31c}
      \sup_{s\in I} \norm{a_{t_i}u(\phi(s))v_i}\leq C\norm{v_i},\quad
        \forall i\in\N.
 \end{equation}
 Therefore we conclude that for any $s\in I$,
 \begin{equation}
   \label{eq:100}
   \pi_+(u(\phi(s)v))=\lim_{i\to\infty} \pi_+(u(\phi(s))v_i)=0.
 \end{equation}

 In other words,
\begin{equation}
  \label{eq:101}
  u(\phi(s))v\subset V^-+V^0,\quad \forall s\in I.
\end{equation}
Then by Corollary~\ref{cor:curve-main}, $v$ is fixed by $G$. But this
contradicts our hypothesis and the proof is complete.
\end{proof}

\section{Ratner's theorem and dynamical behaviour of translated
  trajectories near singular sets} \label{sec:Ratner}

Our aim is to prove that $\mu$, as obtained in
Corollary~\ref{cor:mu-return}, is $L$-invariant. As explained in
\S\ref{sec:sketch}, we will use a technique from \cite{Shah:son1}.

\subsection{Twisted curves and limit measure}
Let $\dot\phi(s)$ denotes the tangent to the curve $\phi$ at $s$. Now
onwards we shall assume that $\dot\phi(s)\neq 0$ for all $s\in I$.
  
Fix $w_0\in\R^{n-1}\setminus \{0\}$, and define
\begin{equation}
  \label{eq:23}
  W=\{u(sw_0):s\in\R\}.
\end{equation}

Recall that $\cen(A)$ acts on $\R^{n-1}$ via the correspondence
$u(z\cdot v)=zu(v)z\inv$ for all $z\in\cen(A)$ and $v\in
\R^{n-1}$. This action is transitive on
$\R^{n-1}\setminus\{0\}$. Therefore there exists an analytic function
$z:I\to \cen(A)$ such that
\begin{equation}
  \label{eq:17}
  z(s)\cdot \dot\phi(s)=w_0.
\end{equation}

For any $i\in\N$, let $\lambda_i$ be the probability measure on
$\lml$ defined by
\begin{equation}
  \label{eq:26}
  \int_{\lml} f\,d\lambda_i:=\frac{1}{\abs{I}}\int_{s\in I}
  f(z(s)a_iu(\phi(s))x_i)\,ds,\quad\forall f\in\Cc(\lml).
\end{equation}

\begin{corollary}
  \label{cor:limit}
  After passing to a subsequence, $\lambda_i\to \lambda$ with respect
  to the weak-$\ast$ topology on the space of probability measures on
  $\lml$.
\end{corollary}

 \begin{proof}
   Given $\epsilon>0$, by Theorem~\ref{thm:return} there exists a
   compact set $\cF\subset \lml$ be such that $\mu_i(\cF)\geq
   1-\epsilon$ for all $i>0$. Since $z(I)$ is compact, there exists a
   compact set $\cF_1\supset z(I)\cF$. Then $\lambda_i(\cF_1)\geq
   1-\epsilon$ for all $i$. Now the corollary is deduced by standard
   arguments using the one-point compactification of $\lml$.
 \end{proof}

 \subsection{Invariance under unipotent flow}

 \begin{proposition}
   \label{prop:invariant}
   Suppose that $\lambda_i{\toinfty}\lambda$ in the space of
   probability measures on $\lml$ with respect to the weak-$\ast$
   topology. Then $\lambda$ is $W$-invariant.
 \end{proposition}

\begin{proof}
  This statement can be deduced by an argument identical to that of
  the proof of \cite[Theorem~3.1]{Shah:son1}. An idea of this proof is
  based on the explanation related to
  \eqref{eq:113}--\eqref{eq:unip}.
\end{proof}

Now we shall describe the measure $\lambda$ using
Ratner's~\cite{R:measure} description of ergodic invariant measures
for unipotent flows. Let $\pi:L\to L/\Lambda$ denote the natural
quotient map. Let $W$ be as defined in \eqref{eq:23}.  For $H\in\sH$,
define
\begin{align*}
  N(H,W)&=\{g\in G: g\inv Wg\subset H\} \qquad\textrm{and}\\
  S(H,W)&=\tcup{\substack{F\in\sH\\F\subsetneq H}}N(F,W).
\end{align*}

Then by Ratner's theorem~\cite{R:measure}, as explained in
\cite[Theorem 2.2]{Moz+Shah:limit}:
\begin{theorem}[Ratner]
  \label{thm:Ratner}
  Given the $W$-invariant probability measure $\lambda$ on $\lml$,
  there exists $H\in\sH$ such that
  \begin{equation}
    \label{eq:lambda-H} 
    \lambda(\pi(N(H,W))>0 \quad\textrm{and} \quad \lambda(\pi(S(H,W))=0.
  \end{equation}
  Moreover almost every $W$-ergodic component of $\lambda$ on
  $\pi(N(H,W))$ is a measure of the form $g\mu_H$, where $g\in
  N(H,W)\setminus S(H,W)$ and $\mu_H$ is a finite $H$-invariant
  measure on $\pi(H)\cong H/H\cap\Lambda$. In particular if $H$ is a
  normal subgroup of $L$ then $\lambda$ is $H$-invariant.\qed
\end{theorem}

\subsection{Algebraic criterion for non-accumulation of measure on
  singular set}

Let $\sA=\{v\in V: v\wedge w_0=0\in \wedge^{d+1}\la{g}\}$. Then $\sA$
is the image of a linear subspace of $V$. We observe that
\begin{equation}
  \label{eq:A}
  N(H,W)=\{g\in L: g\cdot p_H\in \sA\}.
\end{equation}

\begin{proposition}[Cf.~\cite{Dani+Mar:limit}]
  \label{prop:Phi-Psi}
  Given a compact set $\cC\subset \sA$ and $\epsilon>0$, there exists
  a compact set $\cD\subset\sA$ containing $\cC$ such that given any
  neighbourhood $\Phi$ of $\cD$ in $V$, there exists a
  neighbourhood $\Psi$ of $\cC$ in $V$ contained in $\Phi$ such
  that for any $h\in G$, any $v\in V$ and any open interval
  $J\subset I$, one of the following holds:
  \begin{enumerate}
  \item[I)]$hz(t)u(\phi(t))v\in \Phi$ for all $t\in J$.
  \item[II)]
    $\abs{\{t\in J: hz(t)u(\phi(t))v\in\Psi\}} \leq \epsilon 
      \abs{\{t\in J:hz(t)u(\phi(t))v\in\Phi\}}$.
  \end{enumerate}
\end{proposition}

\begin{proof}
  As noted in \cite[Proposition~4.4]{Shah:son1}, the argument in the
  proof of \cite[Proposition 4.2]{Dani+Mar:limit} goes through with
  straightforward changes using the Proposition~\ref{prop:relative}
  instead of \cite[Lemma~4.1]{Dani+Mar:limit}.
\end{proof}

The next criterion is the main outcome of the linearization technique.

\begin{proposition}[Cf.~\cite{Moz+Shah:limit}]
  \label{prop:main3}
  Let $C$ be any compact subset of $N(H,W)\setminus S(H,W)$. Let
  $\epsilon>0$ be given. Then there exists a compact set
  $\cD\subset\sA$ such that given any neighbourhood $\Phi$ of $\cD$ in
  $V$, there exists a neighbourhood $\cO$ of $\pi(C)$ in $\lml$
  such that for any $h_1,h_2\in L$, and a subinterval $J\subset I$,
  one of the following holds:
  \begin{enumerate}
  \item[a)] There exists $\gamma\in\Lambda$ such that
    $(h_1z(t)u(\phi(t))h_2\gamma)p_H\in\Phi$, $\forall t\in J$.
  \item[b)] $\abs{\{t\in J: h_1z(t)u(\phi(t))\pi(h_2)\in\cO\}}\leq
    \epsilon\abs{J}$.
  \end{enumerate}
\end{proposition}

\begin{proof}
  Again this result and its proof are essentially same as those of
  \cite[Proposition 4.5]{Shah:son1}.
\end{proof}

\subsection{Applying the criterion and the basic lemma}

\begin{theorem} 
  \label{thm:equi}
  Suppose that $\lambda_i\toinfty \lambda$ in the space of probability
  measures on $\lml$ with respect to weak-$\ast$ topology. Then
  $\lambda$ is $L$-invariant.
\end{theorem}

\begin{proof}
  By Proposition~\ref{prop:invariant}, $\lambda$ is invariant under
  the action of the nontrivial unipotent subgroup $W$. Therefore by
  Theorem~\ref{thm:Ratner} there exists $H\in\sH$ such that
  \begin{equation}
    \label{eq:28}
    \lambda(\pi(N(H,W))>0 \quad\textrm{and} \quad \lambda(\pi(S(H,W))=0.
  \end{equation}

  Let $C$ be a compact subset of $N(H,W)\setminus S(H,W)$ such that
  $\lambda(C)>\epsilon$ for some $\epsilon>0$. In other words, if we
  write $x_0=\pi(g_0)$ for some $g_0\in G$, then there exists a
  sequence $g_i\to g_0$ such that $x_i=\pi(g_i)$. Given any
  neighbourhood $\cO$ of $\pi(C)$ in $\lml$, there exists $i_0>0$ such
  that for all $i\geq i_0$, we have $\lambda_{i}(\cO)>\epsilon$ for
  all $i>i_0$ and hence
  \begin{equation}
    \label{eq:29}
    \frac{1}{\abs{I}}\abs{\{s\in I:z(s)a_{t_i}u(\phi(s))x_i = 
      \pi(a_{t_i}z(s)u(\phi(s))g_i)\in\cO\}} > \epsilon.
  \end{equation}

  Let $\cD\subset \sA$ be as in the statement of
  Proposition~\ref{prop:main3}. Choose any compact neighbourhood
  $\Phi$ of $\cD$ in $V$. Then there exists a neighbourhood $\cO$ of
  $\pi(C)$ in $\lml$ such that one of the two possibilities of the
  Proposition~\ref{prop:main3} holds. Therefore due to \eqref{eq:29},
  for all $i>i_0$ there exists $\gamma_i\in\Lambda$ such that
  \begin{equation}
    \label{eq:30}
    (z(s)a_{t_i}u(s)g_i\gamma_i)p_H=(a_{t_i}z(s)u(s)g_i\gamma_i)p_H\in
    \Phi, \quad\forall s\in I.
  \end{equation}
  Let $\Phi_1=\{z(s)\inv:s\in I\}\Phi$. Then $\Phi_1$ is contained in
  a compact subset of $V$, and the following holds:
\begin{equation}
  \label{eq:32}
  a_{t_i}u(s)(g_i\gamma_i)p_H\subset \Phi_1, \quad \forall s\in I,\
  \forall i>i_0. 
\end{equation}

Now we express $V=W_0\oplus W_1$, where $W_0$ is the subspace consisting
of all $G$-fixed vectors and $W_1$ is its $G$-invariant
complement. For $i\in\{0,1\}$, let $P_i:V\to W_i$ denote the
associated projection. Consider any norm $\norm{\cdot}$ on $V$ such
that
  \begin{equation}
    \label{eq:33}
    \norm{w}=\max\{\norm{P_0(w)},\norm{P_1(w)}\},\quad\forall w\in V.
  \end{equation}
  Let $R=\sup\{\norm{w}:w\in \Phi_1\}$. By \eqref{eq:32}, for all
  $i\geq i_0$ and $s\in I$ we have
\begin{equation}
  \label{eq:34}
  \norm{{a_t}u(\phi(s))(g_i\gamma_ip_H)} =
  \norm{P_0(g_i\gamma_ip_H)}+\norm{{a_t}u(\phi(s))P_1(g_i\gamma_ip_H)}<R.
\end{equation}
Therefore, by Corollary~\ref{cor:expand} applied to $W_1$ in place of
$V$ and $C=1$, there exists $i_1>i_0$ such that
\begin{equation}
  \label{eq:35}
\norm{P_1(g_i\gamma_ip_H)}<R, \quad\forall i>i_1.
  \end{equation}
  Combining \eqref{eq:33}, \eqref{eq:34} and \eqref{eq:35} we have
\begin{equation}
  \label{eq:37}
  \norm{g_i\gamma_ip_H}<R,\quad\forall i\geq i_1.
\end{equation}
The orbit $\Lambda\cdot p_H$ is discrete due to
Proposition~\ref{prop:discrete}. And $g_i\to g_0$ as
$i\to\infty$. Therefore by passing to a subsequence we may assume that
$\gamma_ip_H=\gamma_{i_1}p_H$ for all $i\geq i_1$.  Put
$\delta_0=\norm{P_1(g_0\gamma_{i_1}p_H)}>0$ and
$C=2R\delta_0\inv$. Then By Corollary~\ref{cor:expand}, there exists
$i_2\geq i_1$ such that for all  $i\geq i_2$ we have
\begin{equation}
  \label{eq:38}
  \sup_{s\in I} \norm{a(t_i)u(\phi(s))P_1(g_i\gamma_{i_1}p_H)}
  \geq C\norm{P_1(g_i\gamma_{i_1}p_H)}>R.
\end{equation}

This contradicts \eqref{eq:34} for all $i\geq i_2$, unless
$P_1(g_0\gamma_{i_1}p_H)=0$. Hence $g_0\gamma_{i_1}p_H$ is
$G$-fixed. Since $\Lambda\cdot p_H$ is closed in $V$, 
$\Lambda\Stab(p_H)=\Lambda\noroneL(H)$ is closed in $L$. Therefore by
taking the inverse $\noroneL(H)\Lambda$ is closed in $L$. Hence the
orbit $\pi(\noroneL(H))$ is closed in $\lml$. By
\cite[Thm.~2.3]{Shah:uniform} there exists a closed subgroup $H_1$ of
$\noroneL(H)$ of $L$ containing all $\Ad$-unipotent one-parameter
subgroups of $L$ contained in $\noroneL(H)$ such that $H_1\cap\Lambda$
is a lattice in $H_1$ and $\pi(H_1)$ is closed. Now $\rho(G)$ is
generated by unipotent one-parameter subgroups. Therefore if we put
$F=g_0\gamma_{i_1}H_1(g_0\gamma_{i_1})\inv$, then $\rho(G)\subset
F$. Also $Fx_0=g_0\gamma_{i_1}\pi(H_1)$ is closed and admits a finite
$F$-invariant measure. Hence by our hypothesis in the beginning of
Section~\ref{sec:nondiv}, $F=L$. Therefore $L=H_1\subset
\noroneL(H)$. Therefore, since $N(H,W)\neq 0$, we have $N(H,W)=L$. In
particular, $W\subset H$. Thus $H\cap\rho(G)$ is a normal subgroup of
$\rho(G)$ containing $W$. Since $\rho(G)$ is a simple Lie group,
$\rho(G)\subset H$. Since $H$ is a normal subgroup of $L$ and $\pi(H)$
is a closed orbit with finite $H$-invariant measure, every orbit of
$H$ on $\lml$ is closed and admits a finite $H$-invariant
measure. Hence by our hypothesis, $H=L$. Now in view of \eqref{eq:28},
by Theorem~\ref{thm:Ratner} $\lambda$ is $H=L$-invariant.
\end{proof}

\begin{corollary}
  \label{cor:mu}
  The measure $\mu$ as in the statement of
  Corollary~\ref{cor:mu-return} is the unique $L$-invariant
  probability measure on $\lml$.
\end{corollary}

\begin{proof}
  Since $\phi$ is analytic, the set of points where $\dot\phi(s)=0$ is
  finite. Therefore it is enough to prove the theorem separately for
  each closed subinterval $J$ of $I$, in place of $I$, under the
  additional hypothesis that $\dot\phi(s)\neq 0$ for all $s\in
  J$. Since $\phi$ is analytic, if $J$ is any subinterval of $I$ with
  nonempty interior, then $\phi(J)$ is not contained in a proper
  affine subspace of $\R^{n-1}$. Therefore without loss of generality
  we assume that $\dot\phi(s)\neq 0$ for all $s\in I$. Let $z(s)\in
  \cen(A)$ be as defined in \eqref{eq:17}. Given $\epsilon>0$ there
  exists a neighbourhood ${O}$ of the $e$ in $\cen(A)$ such that
  $\abs{f(zx)-f(x)}<\epsilon$ for all $x\in \lml$ and $z\in{O}$. We
  consider a partition $I=J_1\cup\dots\cup J_k$ such that for any
  $s,s'\in J_j$, we have $z(s)\inv z(s')\in{O}$. For each
  $j\in\{1,\dots,k\}$, choose $s_j\in J_j$, and define the function
  $f_j(x)=f(z(s_j)\inv x)$ for all $x\in \lml$. Then by
  Theorem~\ref{thm:equi}, applied to the interval $J_j$ in the place
  of $I$, there exists $i_j>0$ such that for all $i>i_j$, we have
  \begin{equation}
    \label{eq:39}
    \Abs{\int_{J_j} f_j(z(s)a_{t_i}u(\phi(s))x_i)\,ds - \abs{J_j}\int_{\lml}
      f_j(x)\,d\mu_L(x)}\leq \epsilon\abs{J_j}. 
  \end{equation}
  Since $\mu_L$ is $\cen(A)$-invariant,
\begin{equation}
  \label{eq:431}
\int_{\lml} f_j(x)\,d\mu_L(x)=\int_{\lml} f(x)\,d\mu_L(x)=:S_j.
\end{equation}
Now 
\begin{align}
  \label{eq:42}
  &\Abs{\int_{J_j} f(a_{t_i}u(\phi(s))x_i)\,ds-\abs{J_j}S_j}\\
\label{eq:43}
&=\Abs{\int_{J_j} f((z(s)\inv z(s_j))z(s_j)\inv
  z(s)a_{t_i}u(\phi(s))x_i)\,ds-\abs{J_j}S_j}\\
\label{eq:44}
&\leq \Abs{\int_{J_j} f_j(z(s)a_{t_i}u(\phi(s))x_i)\,ds -
  \abs{J_j}S_j}+\epsilon\abs{J_j}\\
\label{eq:45}
&\leq 2\epsilon \abs{J_j},
\end{align}
where \eqref{eq:44} follows from the choice of ${O}$ and the
partition of $I$ into $J_j$'s, and \eqref{eq:45} follows from
\eqref{eq:39} and \eqref{eq:431}. 

Therefore for any $i\geq \max\{i_1,\dots,i_k\}$, we have
\begin{align}
  \label{eq:46}
  \abs{I}\cdot \Abs{\int f\,d\mu_i - \int f\,d\mu_L}
  &= \Abs{\int_{I}f(a_{t_i}u(\phi(s))x_i)\,ds - \abs{I}\int f(x)\,d\mu_L}\\
  & \leq \sum_{j=1}^k\Abs{\int_{J_j} f(a_{t_i}u(\phi(s))x_i)\,ds -
    \abs{J_j}\int f\,d\mu_L}\\
\label{eq:46c}
  &\leq 2\epsilon \sum_{j=1}^k \abs{J_j}\leq 2\epsilon \abs{I}.
\end{align}
This shows that $\mu$ is $L$-invariant.  
\end{proof}

\begin{proof}[Proof of Theorem~\ref{thm:action}]
  By \cite[Theorem~2.3]{Shah:uniform} there exists a smallest
    subgroup $H$ of $L$ containing $\rho(G)$ such that the orbit
    $Hx_0$ is closed and admits a finite $H$-invariant
    measure. Therefore replacing $L$ by $H$ and $\Lambda$ by the
    stabilizer of $x_0$ in $H$, without loss of generality we may
    assume that $H=L$.

    If \eqref{eq:9L} fails to hold then there exist $\epsilon>0$ and a
    sequence $t_i\to\infty$ such that for each $i$,
  \begin{equation}
    \label{eq:98}
    \Abs{\frac{1}{\abs{b-a}}\int_a^b f(\rho(a_{t_i}u(\phi(s)))x_0)\,ds-\int_{\lml}
      f\,d\mu_L} \geq \epsilon.
  \end{equation}
  If we put $x_i=x_0$ for each $i$, then in view of \eqref{eq:13} and
  Corollary~\ref{cor:mu-return}, this statement contradicts
  Corollary~\ref{cor:mu}. 
\end{proof}

\begin{proof}[Proof of Theorem~\ref{thm:uniform}] Note that if the
  theorem fails to hold then, there exist sequences $x_i\to x_0$ in
  $\cK$ and $t_i\to\infty$ in $\R$ such that
\begin{equation}
  \label{eq:99}
  \Abs{\frac{1}{|b-a|}\int_a^b f(a_{t_i}u(\phi(s))x_i)\,ds-\int
    f\,d\mu_G}>\epsilon,
\end{equation}
for all $i$. This statement contradicts Corollary~\ref{cor:mu}.
\end{proof}

\begin{proof}[Proof of Theorem~\ref{thm:curve}]
  Let $\psi_{i,j}(s)$ denote the $i,j$-th coordinate of $\vpsi(s)$ for
  all $s\in I$. By our hypothesis, the set $\{t:\psi_{1,1}(t)=0\}$ is
  finite. Therefore arguing as in the proof of Corollary~\ref{cor:mu},
  without loss of generality we may assume that $\psi_{1,1}(s)\neq 0$
  for all $s\in I:=[a,b]$. Define
\[
\phi(s) =
\left(\frac{\psi_{1,2}(s)}{\psi_{1,1}(s)},\dots,\frac{\psi_{1,n}(s)}{\psi_{1,1}(s)}\right)\in\R^{n-1},
\quad\forall s\in I.
\]
Let $U^-=\{g\in G:a_tga_t\inv\stackrel{t\to\infty}{\longrightarrow}
e\}$. Then there exist continuous curves $\psi_-:I\to U^-$ and
$\psi_0:I\to\cen(A)$ such that
\[
\psi(s)=\psi_-(s)\psi_0(s)u(\phi(s)),\quad\forall s\in I.
\]
We observe that the curve $\{\phi(s):s\in I\}$ is contained in a
proper affine subspace of $\R^{n-1}$ if and only if the curve
$\{(\psi_{1,1}(s))_{1\leq j\leq n}):s\in I\}$ is contained in a proper
subspace of $\R^n$.  Given any $\epsilon>0$ and $f\in\Cc(\lml)$, there
exists $t_0>0$ such that for all $t\geq t_0$ and $x\in\lml$, we have
$\abs{f(a_t\psi_-(s)x)-f(a_tx)}<\epsilon$.  Therefore without loss of
generality we may replace $\psi(s)$ by $\psi_0(s)u(\phi(s))$ for all
$s\in I$ to prove the theorem.

Now we apply the argument of the proof of
Corollary~\ref{cor:mu-return} to $\psi_0(s)$ in place of $z(s)$, and
Theorem~\ref{thm:action} in place of Theorem~\ref{thm:equi}, to
complete the proof of the theorem.
\end{proof}

We shall not provide detailed proofs of other results stated in
\S\ref{sec:variations}, as they can be deduced by following the
general strategy of \cite{Dani+Mar:limit} and the method of this
article.

\end{document}